\documentclass[a4paper]{amsart}

\usepackage{tikz}
\usepackage{pgflibraryarrows}
\usepackage{pgflibrarysnakes}

\usepackage{amsfonts}
\usepackage{amssymb}
\usepackage{amsthm}
\usepackage{amsmath}
\usepackage[all]{xy}

\SelectTips{cm}{}

\theoremstyle{plain}
\newtheorem{thm}{Theorem}[section]
\newtheorem{prop}[thm]{Proposition}
\newtheorem{cor}[thm]{Corollary}
\newtheorem{lem}[thm]{Lemma}

\theoremstyle{definition}

\newtheorem{notation}[thm]{Notation}

\theoremstyle{remark}

\numberwithin{equation}{section}

\newcommand{\sN}{\mathcal{N}}
\newcommand{\sS}{\mathcal{S}}
\newcommand{\sT}{\mathcal{T}}

\newcommand{\bt}{\mathbf{t}}
\newcommand{\bs}{\mathbf{s}}

\newcommand{\br}{\mathbf{r}}
\newcommand{\bbr}{\bar{\mathbf{r}}}

\newcommand{\wL}{\widetilde{L}}
\newcommand{\wsN}{\widetilde{\sN}}
\newcommand{\wrho}{\widetilde{\rho}}

\newcommand{\del}{\partial}

\newcommand{\CC}{\mathbb{C}}

\newcommand{\QQ}{\mathbb{Q}}

\newcommand{\ZZ}{\mathbb{Z}}

\newcommand{\nin}{\noindent}

\newcommand{\ra}{\rightarrow}
\newcommand{\xra}{\xrightarrow}

\newcommand{\co}{\colon\!}

\newcommand{\BO}{\textup{BO}}

\newcommand{\G}{\textup{G}}
\newcommand{\TOP}{\textup{TOP}}
\newcommand{\Gsign}{\textup{G-sign}}
\renewcommand{\mod}{\textup{mod }}
\newcommand{\red}{\textup{red}}

\newcommand{\RhG}{R_{\widehat G}}
\newcommand{\RhGp}{R_{\widehat G}^+}
\newcommand{\RhGm}{R_{\widehat G}^-}

\newcommand{\proj}{\textup{proj}}

\title[Higher simple structure sets of lens spaces]{Higher simple structure sets of lens spaces \\ with the fundamental group of arbitrary order}
\author{L'udov\'it Balko}
\author{Tibor Macko}
\author{Martin Niepel}
\author{Tom\'a\v s Rusin}



\address{Faculty of Mathematics, Physics, and Informatics, Comenius University, Mlynsk\'a dolina,
SK-824 48, Bratislava, Slovakia} \email{ludovit.balko@fmph.uniba.sk} \email{tibor.macko@fmph.uniba.sk} \email{martin.niepel@fmph.uniba.sk} \email{tomas.rusin@fmph.uniba.sk} 

\address{Institute of Mathematics, Slovak Academy of Sciences, \v Stef\'anikova 49, Bratislava, SK-81473, Slovakia} \email{macko@mat.savba.sk}

\thanks{This work was supported by the grant VEGA 1/0101/17, and by the Slovak Research and Development Agency under the contract No. APVV-16-0053.}

\date{\today}

\begin{document}

\maketitle

\begin{abstract}
	Extending work of many authors we calculate the higher simple structure sets of lens spaces in the sense of surgery theory with the fundamental group of arbitrary order. As a corollary we also obtain a calculation of the simple structure sets of the products of lens spaces and spheres of dimension grater or equal to $3$.
\end{abstract}



\section*{Introduction}
\label{sec:intro}


This paper is to be considered a second part of~\cite{BMNR(2018)}. There we calculated the topological higher simple structure sets $\sS^{s}_{\del} (L \times D^{m})$ in the sense of surgery theory for $L$ a fake lens space with the fundamental group $\ZZ/N \cong G = \pi_{1} (L)$ for $N = 2^{K}$ and $m \geq 1$. The main result of the present paper, Theorem~\ref{thm:main-thm}, calculates $\sS^{s}_{\del} (L \times D^{m})$ when $\ZZ/N \cong G = \pi_{1} (L)$ has $N$ arbitrary and $m \geq 1$.

Let us recall that the case $m=0$ was done for $N=M$ odd in \cite[14.E]{Wall(1999)}, for $N=2$ in \cite[14.D]{Wall(1999)} and \cite{LdM(1971)}, and for $N=2^K$ with $K \geq 2$ in \cite{Macko-Wegner(2009)}. These results were then combined in \cite{Macko-Wegner(2011)} to obtain the general case $N=2^K \cdot M$ with $M$ odd.

For the case $m \geq 1$ calculations for $N=M$ odd and $N=2$ appeared in \cite{Madsen-Rothenberg(1989)}. As already mentioned, the case $N=2^K$ is done in \cite{BMNR(2018)}. The present paper combines these results to obtain the general case $N=2^K \cdot M$ with $M$ odd. We follow the same general idea about the combination as in \cite{Macko-Wegner(2011)}, but we need to make some adjustments. Recall that the key part of the calculations concerns in all of the above papers the kernel of the so-called $\rho$-invariant on the group of normal invariants. When $m=0$ all the groups of normal invariants are finite. However, when $m \geq 1$ this is no longer the case. This is the main technical obstacle that we address in this paper. Roughly speaking, our key idea in doing so is to use the formula for the $\rho$-invariant from Theorem 5.5  in~\cite{BMNR(2018)}, which we proved in the general case $N =2^K \cdot M$ with $M$ odd, to show that the $\rho$-invariant map in all the cases factors through a finite group, see Lemmas~\ref{lem:rho-factors-through-finite-2-K}, \ref{lem:rho-factors-through-finite-M}, \ref{lem:rho-factors-through-finite-N}. We then obtain the desired kernel by combining the kernels of the two factoring maps. One of them is obvious and the other is obtained by the method of \cite{Macko-Wegner(2011)} since its source is finite, see Theorem~\ref{thm:K-bar-N} and Corollary~\ref{cor:T-N}.

The paper is organized as follows. In Section~\ref{sec:results} we state the main Theorem~\ref{thm:main-thm} and Corollary~\ref{cor:higehr-str-sets-L-times-S}. These are deliberately formulated to be analogues of the main results from~\cite{BMNR(2018)} where $N = 2^{K} \cdot M$ with $M$ odd in the present case. In Section 2 we describe the changes to the setup of~\cite{BMNR(2018)} for the normal invariants that are necessary to accommodate the present case, and in Sections 3 and 4 we present the techniques and calculations for the proof of the main theorem.


\section{Results}
\label{sec:results}


A {\it fake lens space} $L = L (\alpha)$ is a topological manifold given as the orbit space of a free action $\alpha$ of a finite cyclic group $G = \ZZ/N$ on a sphere $S^{2d-1}$. The main result of this paper is the following theorem about them.

\begin{thm}\label{thm:main-thm}
	Let $L = L (\alpha)$ be a $(2d-1)$-dimensional fake lens space for some free action $\alpha$ of the cyclic group $G = \ZZ/N$ with $N = 2^K \cdot M$ for some $K \geq 1$ and $M$ odd on $S^{2d-1}$ with $d \geq 2$ and let $k \geq 1$. Then we have isomorphisms
	\begin{align*}
		(\wrho_{\del},\bbr_{0},\bbr,\br) \co \sS^{s}_{\del} (L \times D^{2k}) & \xra{\cong} 
		\begin{cases}
			F^{+} \oplus \ZZ \oplus T'_{2^K} \oplus T_{2} \quad & d = 2e, k=2l \\
			F^{-} \oplus \ZZ/2 \oplus T'_{2^K} \oplus T_{2} \quad & d = 2e, k=2l+1 \\
			F^{-} \oplus \ZZ \oplus T'_{2^K} \oplus T_{2} \quad & d = 2e+1, k=2l \\
			F^{+} \oplus \ZZ/2 \oplus T'_{2^K}  \oplus T_{2} \quad & d = 2e+1, k=2l+1 
		\end{cases} \\
		(\bbr,\br) \co \sS^{s}_{\del} (L \times D^{2k+1}) & \xra{\cong} T_{2} (\textup{odd}) \quad \textup{also} \; k=0,
	\end{align*}
	where the meaning of the symbols in the target is as follows:
	\begin{enumerate}
		\item $F^{+}$ is a free abelian group of rank $2^{K-1} \cdot M$;
		\item $F^{-}$ is a free abelian group of rank $2^{K-1} \cdot M - 1$;
		\item Let $c_N (d,k) = e-1$ when $(d,k)=(2e,2l)$ and let $c_N (d,k) = e$ in other cases. Then 
		 \[ 
		 T'_{2^K} \cong \bigoplus_{i=1}^{c_N(d,k)} \ZZ/2^{\textup{min} \{ 2i , K \}}; 
		 \]
		\item Let  $c_2 (d,k) = e$ when $(d,k)=(2e+1,2l)$ and let $c_2 (d,k) = e-1$ in other cases. Then 
		 \[ 
		 T_{2} \cong \bigoplus_{i=1}^{c_2 (d,k)} \ZZ/2; 
		 \]  
		\item Let  $c_2 (d,k,\textup{odd}) = e-1$ when $(d,k)=(2e,2l+1)$ and let $c_2 (d,k,\textup{odd}) = e$ in other cases. Then
		 \[ 
		 T_{2} (\textup{odd}) \cong \bigoplus_{i=1}^{c_2 (d,k,\textup{odd})} \ZZ/2; 
		 \] 
		\item The symbol $\wrho_{\del}$ denotes the reduced $\rho$-invariant for manifolds with boundary;
		\item The invariant $\bbr$ is an invariant derived from the splitting invariants along $4i$-dimensional submanifolds;
		\item The invariant $\br$ consists of the splitting invariants along $(4i-2)$-dimensional submanifolds;
		\item The invariant $\bbr_{0}$ is an invariant derived from the splitting invariants along $2k$-dimensional submanifolds.
	\end{enumerate}
\end{thm}

The definitions of the invariants $\wrho_{\del}$, $\br$, $\bbr_{0}$ and $\bbr$ are taken from~\cite{BMNR(2018)}.

\begin{cor} \label{cor:higehr-str-sets-L-times-S}
	For $k \geq 2$ we have isomorphisms 
	\[	
		(\red{_\del},\br',\br'',\br''') \co \sS^{s} (L \times S^{2k}) \cong \sS^{s}_{\del} (L \times D^{2k}) \oplus T_{2^K} (d) \oplus T_2 (d) \oplus T_M (d) 
	\]
	and for $k \geq 1$ we have isomorphisms 
	\[	
		(\red{_\del},\br',\br'',\br''') \co\sS^{s} (L \times S^{2k+1}) \cong \sS^{s}_{\del} (L \times D^{2k+1}) \oplus T_{2^K} (d) \oplus T_2 (d) \oplus T_M (d)
	\]
	where for $c = \lfloor (d-1)/2 \rfloor$ 
	\[
	T_{2^K} (d) \cong \bigoplus_{i=1}^{\lfloor (d-1)/2 \rfloor} \ZZ/2^{K} \quad \textup{and} \quad T_2 (d) \cong \bigoplus_{i=1}^{\lfloor d/2 \rfloor} \ZZ/2 \quad \textup{and} \quad |T_M (d)| = M^c.
	\]
\end{cor}
The definitions of the map $\red_\del$ and the invariants $\br',\br''$ are taken from~\cite{BMNR(2018)}, the invariant $\br'''$ is addressed in Section~\ref{sec:ses}. Together with Theorem \ref{thm:main-thm} this shows that $\sS^{s} (L \times S^{m})$ is calculated by the invariants $\wrho_{\del},\bbr_{0},\bbr,\br,\br', \br'',\br'''$ where each symbol has to be appropriately interpreted depending on parity of $d$ and $m$.


\section{The surgery exact sequence}
\label{sec:ses}


The basic setup of the paper \cite{BMNR(2018)} fits the general case handled here as well. In particular,  the material of Sections 2 and 3 of \cite{BMNR(2018)} serves as a background and motivation also in our case. Similarly, Section 5 about the $\rho$-invariant is written for the general case. Hence we refer the reader there for these topics as well as for the notation. Modifications are needed in Section 4 and, of course, in Section 6 which contains calculations. Hence, we concentrate on these in this paper. 


As in~\cite{BMNR(2018)} it is enough to deal with the case $L^{2d-1} = L^{2d-1}_{(1,\ldots,1)}$. We recall that we need to study the surgery exact sequence:
\begin{equation} \label{eqn:ses-lens-times-disk}
\sN_\partial (L \times D^{m+1}) \xra{\theta} L^s_{n+1} (\ZZ G) \xra{\partial} \sS_\del^s (L \times D^{m}) \xra{\eta} \sN_\partial (L \times D^{m}) \xra{\theta} L^s_{n} (\ZZ G),
\end{equation}
where $n=\dim (L \times D^{m}) = 2d-1+m$.

For the $L$-groups we have the following theorem where the symbol $R_{\CC} (G)$ denotes the complex representation ring of a group $G$ and the superscripts $\pm$ denote the $\pm$-eigenspaces with respect to the involution given by complex conjugation. The symbol $\Gsign$ means the $G$-signature and Arf is the Arf invariant.

\begin{thm} \label{L(G)} \cite{Hambleton-Taylor(2000)}
	For $G = \ZZ_N$ with $N = 2^K \cdot M$, $M$ odd, we have that
	\begin{align*}
	L^s_n (\ZZ G) & \cong
	\begin{cases}
	4 \cdot R_{\CC}^+ (G) & n \equiv 0 \; (\mod 4) \; (\Gsign, \;
	\mathrm{purely} \; \mathrm{real}) \\
	0 & n \equiv 1 \; (\mod 4) \\
	4 \cdot R_{\CC}^- (G) \oplus \ZZ/2 & n \equiv 2 \; (\mod 4) \;
	(\Gsign, \; \mathrm{purely}
	\; \mathrm{imaginary}, \mathrm{Arf}) \\
	0 \; \textup{or} \; \ZZ/2 & n \equiv 3 \; (\mod 4) \;
	(\mathrm{codimension} \; 1 \; \mathrm{Arf})
	\end{cases}
	\end{align*}
	$\widetilde L^s_{2k} (\ZZ G) \cong 4 \cdot \RhG^{(-1)^k}$ where
	$\RhG^{(-1)^k}$ is $R_{\CC}^{(-1)^k} (G)$ modulo the regular
	representation. In the case $n \equiv 3 \; (\mod 4)$ we have $0$ if
	$K=0$ and $\ZZ/2$ if $K \geq 1$.
\end{thm}

For the normal invariants~$\sN_{\del} (Y) \cong [Y/\del Y;\G/\TOP]$, we first notice that for the products $Y = L \times D^{m}$ we have
\begin{equation} \label{eqn:l-times-d-is-wedge-with-smash}
Y/\del Y = L \times D^{m} / L \times S^{m-1} \simeq L_{+} \wedge S^{m} \simeq (L \wedge S^{m}) \vee S^{m}.
\end{equation}
Using localization at $2$ and away from $2$, we have in general the following homotopy pullback square \cite{Madsen-Milgram(1979)}
\begin{equation} \label{eqn:htpy-type-of-g-top}
\begin{split}
\xymatrix{
	\G/\TOP \ar[r] \ar[d] & \prod_{i > 0} K(\ZZ_{(2)},4i) \times K(\ZZ/2,4i-2) \ar[d] \\
	\BO[1/2] \ar[r] & \BO_{\QQ} \simeq \prod_{i > 0} K(\QQ,4i)
}
\end{split}
\end{equation}
which induces a Mayer-Vietoris sequence for the homotopy sets of mapping spaces. 

It is a good idea to combine both of these results in a way that suits a particular purpose. Therefore,  slightly differently from~\cite{BMNR(2018)}, we first use~\eqref{eqn:l-times-d-is-wedge-with-smash} to obtain
\begin{equation} \label{eqn:ni-lens-space-times-disk-v1}
\sN_{\del} (L \times D^{m}) \cong [S^{m},\G/\TOP] \oplus [L \wedge S^{m},\G/\TOP].
\end{equation}
We also have 
\[
[L \wedge S^{m},\G/\TOP] \cong H^{-m} (L,\G/\TOP).
\]
Hence we have
\begin{align} \label{eqn:ni-calculation-via-sing-co-and-ko}
\begin{split}
& \sN_{\del} (L_{N} \times D^{m}) \cong \\ & L_{m} (\ZZ) \oplus_{i} H^{4i-m} (L_{N};\ZZ_{(2)}) \oplus_{i} H^{4i-2-m} (L_{N};\ZZ_2) \oplus \widetilde{KO}^{-m} (L_{N};\ZZ [1/2]). 
\end{split}
\end{align}
We have enough information about all these groups, see 
Theorem 3.2 from \cite{Macko-Wegner(2011)}. It is convenient to distinguish the two cases when $m$ is odd and when $m$ is even.

\

\nin \textbf{Case $m=2k+1$.} 
The Atiyah-Hirzebruch spectral sequence reveals that 
\begin{equation} \label{ni-lens-space-times-odd-disk}
\sN_{\del} (L \times D^{2k+1}) \cong \begin{cases}
\ZZ \oplus \bigoplus_{i \in J_{2}^{N} (d,k,\textup{odd})} \ZZ/2 & \textup{if} \; (d,k)=(2e,2l) \;\textup{or} \; (2e+1,2l+1) \\
\bigoplus_{i \in J_{2}^{N} (d,k,\textup{odd})} \ZZ/2 & \textup{otherwise},
\end{cases}
\end{equation}
where the indexing set $J_{2}^{N} (d,k,\textup{odd})$ is as in (4.3) of~\cite{BMNR(2018)}. Note that when $(d,k)=(2e,2l)$, then $2d-1+2k+1 = 4(e+l)$, and when $(d,k)=(2e+1,2l+1)$, then $2d-1+2k+1=4(e+l+1)$. The $\ZZ/2$ summands are detected by the invariants $\bt_{4i-2}$, until we reach the dimension of $2d-1+2k+1$.

\

\nin \textbf{Case $m=2k$.} From~\eqref{eqn:ni-calculation-via-sing-co-and-ko} we see that this case is a shifted copy of the case $k=0$ plus a summand coming from the sphere $S^{2k}$. We also see that for $N=2^{K} \cdot M$ with $M$ odd the singular cohomology part is the same as for $2^{K}$ and $M$ only influences the $KO$-theory part. From the Atiyah-Hirzebruch spectral sequence one sees that this part has order $M^{c}$ where $c = \lfloor (d-1)/2 \rfloor$. 
Hence we get
\begin{equation} \label{red-ni-lens-spaces}
\sN_{\del} (L_{N} \times D^{m}) \cong L_{m} (\ZZ) \oplus \bigoplus_{i} \ZZ/{2^K}
\oplus \bigoplus_{i} \ZZ/2 \oplus \widetilde{KO}^{-m} (L;\ZZ [1/2])
\end{equation}
where the order of the last summand is $M^c$ with $c = \lfloor (d-1)/2 \rfloor$. As in~\cite{BMNR(2018)} all but the last summand are detected by the splitting invariants $\bt_{4i}$ and $\bt_{4i-1}$, where the index $i$ runs through indexing sets denoted as $J_{4}^{N} (d,k)$, $J_{2}^{N} (d,k)$.

\

Proposition 4.2 from \cite{BMNR(2018)}, which influences both cases of $m$, is valid for general $2 |  N$. If $N=M$ is odd, then $L^s_{2d-1+2k}(\ZZ G) =0$. Therefore we denote here as well
\begin{equation}
\widetilde{\sN}_{\del} (L \times D^{m}) := \ker \theta \co \sN_{\del}(L^{2d-1} \times D^{m}) \ra L^s_{2d-1+m}(\ZZ G)
\end{equation}
and the corresponding indexing sets as $J_{4}^{tN} (d,k)$, $J_{2}^{tN} (d,k)$ and $J_{2}^{tN} (d,k,\textup{odd})$.

\

We can now summarize what we know. Our information is enough to solve the case $m=2k+1$, the other case will take more effort.

\

\nin \textbf{Case $m=2k+1$.} We have the isomorphism
\begin{equation} \label{eqn:end-result-odd-disk}
\sS^{s}_{\del} (L \times D^{2k+1}) \cong \widetilde{\sN}_{\del} (L \times D^{2k+1}) \cong \bigoplus_{J_{2}^{tN} (d,k,\textup{odd})} \ZZ/2.
\end{equation}
This follows from~\eqref{ni-lens-space-times-odd-disk} combined with Proposition 4.2 from \cite{BMNR(2018)}.

\

\nin \textbf{Case $m = 2k$.} We obtain the short exact sequence
\begin{equation} \label{ses-lens-2d-1}
0 \ra \wL^s_{2d+2k} (\ZZ G) \xra{\partial} \sS^{s}_{\del}  (L^{2d-1} \times D^{2k})
\xra{\eta} \widetilde{\sN}_{\del}(L^{2d-1} \times D^{2k}) \ra 0,
\end{equation}
where
\begin{align}
\begin{split} \label{eqn:tilde-N-L-x-D}
n = 4u-1 \; : \; \widetilde{\sN}_{\del} (L^{2d-1} \times D^{2k}) & = \mathrm{ker} \; \big (
\bt_{4u-2} \co {\sN}_{\del} (L^{2d-1} \times D^{2k}) \ra \ZZ/2 \big ), \\
n = 4u+1 \; : \; \widetilde{\sN}_{\del} (L^{2d-1} \times D^{2k}) & = \sN_{\del} (L^{2d-1} \times D^{2k}).
\end{split}	
\end{align}
It will be convenient to use the decomposition
\begin{equation} \label{red-ni-lens-spaces}
\widetilde \sN_{\del} (L^{2d-1} \times D^{2k}) \cong T_{F} (d,k) \oplus T_{2^{K}} (d,k) \oplus T_{2} (d,k) \oplus T_{M} (d,k),
\end{equation}
where 
\[
T_F (d,k) \cong \begin{cases} \ZZ (t_{4l}) & k = 2l \\ 0 & k = 2l+1 \end{cases}
\]
and 
\[ 
T_{2^{K}} (d,k) = \bigoplus_{i \in rJ_{4}^{tN} (d,k)} \ZZ/{2^K} (t_{4i}), \quad  T_2 (d,k) = \bigoplus_{i\in J_{2}^{tN} (d,k)} \ZZ/2 (t_{4i-2})
\]
and 
\[
|T_{M} (d,k)| = M^{c}. 
\]
Here the indexing sets $rJ_{4}^{tN} (d,k)$ and $rJ_{4}^{tN} (d,k)$, $rJ_{2}^{tN} (d,k)$ and $rJ_{2}^{tN} (d,k)$ are as in~\cite{BMNR(2018)} after display (4.11). The cardinality of $rJ_{4}^{tN} (d,k)$ is equal to $c_N (d,k)$ and the cardinality of $rJ_{2}^{tN} (d,k)$ is equal to $c_2 (d,k)$ from the statement of Theorem~\ref{thm:main-thm}.

We will also sometimes use the notation
\begin{equation} \label{eqn:torsion-of-ni-general}
\sT_{N}(d,k) := T_{2^{K}} (d,k) \oplus T_{2} (d,k) \oplus T_{M} (d,k)
\end{equation}
for the torsion part of \eqref{red-ni-lens-spaces}. Note that for $M$ odd $\sT_{M} (d,k) = T_{M} (d,k)$.

The first term in the sequence (\ref{ses-lens-2d-1}) is understood by Theorem \ref{L(G)}, the third
term is understood by (\ref{red-ni-lens-spaces}). Hence we are left with an extension problem, which we solve in Section~\ref{sec:calculations} using techniques from Section~\ref{sec:the-rho-invariant}.


\section{The $\rho$-invariant} \label{sec:the-rho-invariant}


By the preceding section we now understand the surgery exact sequence for $L^{2d-1} \times D^{2k}$ up to an extension problem which we solve by studying the $\rho$-invariant. In Subsections 5.1 to 5.3 of~\cite{BMNR(2018)} its definition and main properties are recalled. These all work in our general case. Hence, just as in~\cite{BMNR(2018)}, we now have all the ingredients we need to analyze the surgery exact sequence for $X = L^{2d-1} \times D^{2k}$. Denoting $n = 2d-1+2k$ we can summarize everything in the commutative ladder:
\[
\xymatrix{
	0 \ar[r] & \wL_{n+1} (\ZZ G) \ar[r] \ar[d]_{G-\textup{sign}} & \sS^{s}_{\del} (L \times D^{2k}) \ar[d]^{\wrho_{\del}} \ar[r] & \wsN_{\del} (L \times D^{2k}) \ar[d]^{[\wrho_{\del}]} \ar[r] & 0 \\
	0 \ar[r] & 4 \cdot \RhG^\pm \ar[r] & \QQ \RhG^{\pm} \ar[r] & \QQ \RhG^{\pm}/4 \cdot \RhG^{\pm} \ar[r] & 0  
} 
\]
By Theorem 5.4 of~\cite{BMNR(2018)} we need to understand the kernel of $[\wrho_\del]$. Since in this paper  we only work with the $\rho$-invariant for manifolds with boundary, but we will be switching between the $\rho$-invariant maps for different groups $G = \ZZ/N$ and different dimensions $(d,k)$, we introduce the following notation.

\begin{notation}
	When $G = \ZZ/N$ we denote the $\rho$-invariant map from Definition 5.3 of~\cite{BMNR(2018)} as
	\[
		\wrho_N (d,k) \co \sS^{s}_\partial (L^{2d-1}_{N} \times D^{2k}) \ra \QQ \RhG^{\pm}
	\]
	and the induced map on the normal invariants as
	\[
		[\wrho_N (d,k)] \co \sN_\partial (L^{2d-1}_{N} \times D^{2k}) \ra \QQ \RhG^{\pm}/4 \cdot \RhG^{\pm}.
	\]
\end{notation}	

\

Hence our task is to calculate $\ker [\wrho_N (d,k)]$. Imitating~\cite{Macko-Wegner(2011)} we will relate it to $\ker [\wrho_{2^K} (d,k)]$ and $\ker [\wrho_{M} (d,k)]$. In order to do that we need to use the transfer maps and their compatibility with the $\rho$-invariant.

Given a natural number $N = U \cdot V$ and $d \geq 1$ let $p_U^N \co L^{2d-1}_{U} \ra L^{2d-1}_{N}$ be the covering\footnote{Here we include the subscript to indicate which lens space is meant.} induced by the inclusion of a subgroup $\ZZ/U \subset \ZZ/N$. These define transfer maps on the normal invariants which are compatible with the $\rho$-invariant so that we have the following commutative diagrams:

\begin{equation}\label{diagram_2^K}
\begin{split}
\xymatrix{
	L_{2k} (\ZZ) \oplus \sT_{N} (d,k) \ar[r]^{\cong} \ar[d] & \widetilde \sN_{\del} (L^{2d-1}_N \times D^{2k}) \ar[d]_{(p_{2^K}^N)^{!}} \ar[r]^-{[\wrho_N (d,k)]} & \QQ R_{\widehat \ZZ_N}^{\pm}/4 \cdot R_{\widehat \ZZ_N}^{\pm} \ar[d] \\
	L_{2k} (\ZZ) \oplus \sT_{2^K} (d,k) \ar[r]^{\cong} & \widetilde \sN_{\del} (L^{2d-1}_{2^K} \times D^{2k})
	\ar[r]^-{[\wrho_{2^K} (d,k)]} & \QQ R_{\widehat \ZZ_{2^K}}^{\pm}/4 \cdot R_{\widehat \ZZ_{2^K}}^{\pm} }
\end{split}
\end{equation}
and
\begin{equation}\label{diagram_M}
\begin{split}
\xymatrix{
	L_{2k} (\ZZ) \oplus \sT_{N} (d,k) \ar[r]^{\cong} \ar[d] & \widetilde \sN_{\del} (L^{2d-1}_N \times D^{2k}) \ar[d]_{(p_M^N)^{!}} \ar[r]^-{[\wrho_N (d,k)]} & \QQ R_{\widehat \ZZ_N}^{\pm}/4 \cdot R_{\widehat \ZZ_N}^{\pm} \ar[d] \\
	L_{2k} (\ZZ) \oplus \sT_M (d,k) \ar[r]^{\cong} & \widetilde \sN_{\del} (L^{2d-1}_M \times D^{2k}) \ar[r]^-{[\wrho_M (d,k)]} & \QQ
	R_{\widehat \ZZ_M}^{\pm}/4 \cdot R_{\widehat \ZZ_M}^{\pm}. }
\end{split}
\end{equation}

In addition, given $N \geq 1$ there is the $S^1$-bundle $p_N^{S^1} \co L^{2d-1}_{N} \ra \CC P^{d-1}$ which also induces a transfer map on the normal invariants compatible with the $\rho$-invariant. 

The key property of the $\rho$-invariant is Theorem 5.5 in~\cite{BMNR(2018)} which we reproduce here for convenience. The indexing set $I_{4}^{S} (d,k)$ in the statement is defined in Section~3 of~\cite{BMNR(2018)}. 

\begin{thm} \label{thm:rho-formula-cp-times-disk}
	Let $h \co Q \ra \CC P^{d-1} \times D^{2k}$ represent an element in the higher structure set $\sS_{\del} (\CC P^{d-1} \times D^{2k})$. Then for $1 \neq t \in S^1$ we have
	\[
	\wrho_{S^1,\del} (d,k)(t,[h]) = \sum_{i \in I_{4}^{S} (d,k)} 8 \cdot \bs_{4i} (\eta ([h])) \cdot (f^{d+k-2i} - f^{d+k-2i-2}) \in \CC 
	\]
	where $f = (1+t)/(1-t)$. 
\end{thm}

Also recall that we showed that the composition
\begin{equation} \label{eqn:fibering-ni-lens-x-disk-by-cp-x-disk}
	\sS_{\del} (\CC P^{d-1} \times D^{2k}) \xra{\eta} \sN_{\del} (\CC P^{d-1} \times D^{2k}) \xra{\textup{proj} \circ (p_{G}^{S^{1}})^{!}} \widetilde{\sN}_{\del} (L^{2d-1} \times D^{2k})
\end{equation}
is surjective for $n-1=2d-2+2k=4u+2$, which can be phrased as saying that any representative of any element in $\sS^{s}_{\del} (L^{2d-1} \times D^{2k})$ is normally cobordant to a representative of possibly another element of the same group which fibers over a fake $\CC P^{d-1} \times D^{2k}$. In case $n-1=2d-2+2k=4u$ this map is close to be surjective, which can be phrased by saying that in case $n-1=2d-2+2k=4u$ the suspension of any element of $\sS^{s}_{\del} (L^{2d-1} \times D^{2k})$ is normally cobordant to a representative of an element of $\sS^{s}_{\del} (L^{2d+1} \times D^{2k})$ which fibers over a fake $\CC P^{d} \times D^{2k}$. Compare to \cite[Lemma 14E.9]{Wall(1999)}

In~\cite{BMNR(2018)} a consequence of these ideas was a formula for the $\rho$-invariant map $[\wrho_{2^{K}} (d,k)]$ in Proposition 5.7. Here we will not need the exact formula, but we will use another consequence, which is formulated as Proposition~\ref{prop:rho-N-d-k}.

To start, we need some notation which is at first unrelated to the above remarks. Let $\ZZ (d,k)$ be the free abelian group defined as follows:
\begin{equation} \label{eqn:Z-d-k-2-mod-4}
\ZZ (d,k) := 
\begin{cases} 
	\oplus_{i \in I_{4}^{S} (d,k)} \ZZ(s_{4i}) & \textup{if} \; n-1=2d-2+2k=4u+2 \\ 
	\oplus_{i \in I_{4}^{S} (d+2,k)} \ZZ(s_{4i}) & \textup{if} \; n-1=2d-2+2k=4u 
\end{cases}
\end{equation}

Define maps
\begin{equation}
\rho_N (d,k) \co \ZZ (d,k) \ra \QQ \RhG^{(-1)^{d+k}} 
\end{equation}
by
\begin{equation} 
	\label{eqn:formula-for-rho-Z-d-k-even}
	\rho_{N} (d,k) ((s_{4i})_{i}) = \sum_{i \in I^{S}_{4} (d,k)} \!\! 8 \cdot s_{4i} \cdot (f^{d+k-2i} - f^{d+k-2i-2}) \in \QQ\RhGp, 
\end{equation}
when  $n-1=2d-2+2k=4u+2$ and by
\begin{equation} 	
	\label{eqn:formula-for-rho-Z-d-k-odd}
	\rho_{N} (d,k) ((s_{4i})_{i}) = \sum_{i \in rI^{S}_{4} (d+2,k)} \!\! 8 \cdot s_{4i} \cdot (f^{d+k-2i} - f^{d+k-2i-2}) + 8 \cdot s_{4u} \cdot f \in \QQ\RhGm
\end{equation}
when $n-1=2d-2+2k=4u$ with $rI^{S}_{4} (d+2,k) = I^{S}_{4} (d+2,k) \smallsetminus \{ u \}$.

Correspondingly, let
\begin{equation} \label{eqn:brackets-rho-N-d-k}
[\rho_N (d,k)] \co \ZZ (d,k) \xra{\rho_{N} (d,k)} \QQ \RhG^{(-1)^{d+k}} \ra \QQ \RhG^{(-1)^{d+k}} / 4 \cdot \RhG^{(-1)^{d+k}}
\end{equation}

Also denote by $\proj_{N} (d,k)$ the map $\ZZ (d,k) \ra \widetilde{\sN}_{\del} (L^{2d-1}_{N} \times D^{2k})$ obtained by including $\ZZ (d,k)$ into $\sS_{\del} (\CC P^{a-1} \times D^{2k})$ for $a=d$ when $n-1=2d-2+2k=4u+2$ and $a=d+2$ when $n-1=2d-2+2k=4u$ and then applying the transfer map. The symbol $\red_{N}^{S^{1}}$ below denotes the appropriately restricted reduction map $R_{\widehat{S}^{1}} \ra R_{\widehat{\ZZ/N}}$ induced by the inclusion $\ZZ/N \subset S^1$.

\begin{prop} \label{prop:rho-N-d-k}
We have
\[
	\red_{N}^{S^{1}} \circ \rho_{S^{1}} (d,k) = [\rho_N (d,k)] = [\wrho_N (d,k)] \circ \proj_{N} (d,k).
\]
\end{prop}

\begin{proof}
	This follows from the formulas for the various $\rho$-maps, Theorem~\ref{thm:rho-formula-cp-times-disk}, displays~\eqref{eqn:formula-for-rho-Z-d-k-even} and~\eqref{eqn:formula-for-rho-Z-d-k-odd}, and the fact that the $\rho$-invariant is natural together with the property described in the paragraph around  equation~\eqref{eqn:fibering-ni-lens-x-disk-by-cp-x-disk}. It can be phrased as saying that the following diagram	
	\[
	\xymatrix{
	\sS_{\del} (\CC P^{a-1} \times D^{2k})	\ar[d]_{\wrho_{S^1} (d,k)} & \ZZ (d,k) \ar[l] \ar[rr]^-{\proj_{N} (d,k)}  \ar[d]_{\rho_N (d,k)} & & \widetilde{\sN}_{\del} (L^{2d-1} \times D^{2k}) \ar[d]^{[\wrho_N (d,k)]} \\
	\QQ R_{\widehat{S}^{1}}^{\pm} \ar[r]_{\red_{N}^{S^{1}}} & \QQ R_{\widehat{\ZZ/N}}^{\pm} \ar[rr] & & \QQ R_{\widehat{\ZZ/N}}^{\pm} / 4 \cdot R_{\widehat{\ZZ/N}}^{\pm}
	}
	\]
	is commutative.
\end{proof}

The meaning of the equation in Proposition~\ref{prop:rho-N-d-k} is that the map $[\wrho_N]$ in whose kernel we are interested is described by the above formulas when precomposed with the map $\proj_{N} (d,k)$ which surjects onto $\sT_{2^K} (d,k) \oplus \sT_{M} (d,k)$.


\section{Calculations} \label{sec:calculations}


As indicated in the introduction, it is convenient to observe that the maps $[\wrho_{N} (d,k)]$ factor through finite groups. In the case $k=2l+1$ the source~\eqref{red-ni-lens-spaces} already is a finite group so there is nothing to do. In the case $k=2l$ we need a proof, which we divide into three cases, $N=2^{K}$, $N=M$ odd, and $N=2^{K} \cdot M$. In all the cases the key technical idea is to look at the corresponding situation in the case when $k=0$ and $d$ increases by $2$, that means we will be looking at calculations of the composition 
\[
	\sS (\CC P^{d+1}) \xra{\textup{proj} \circ (p_{G}^{S^{1}})^{!} \circ \eta} \widetilde{\sN} (L^{2d+3}) \xra{[\wrho_{N} (d+2,0)]} \QQ R_{\widehat{\ZZ/N}}^{(-1)^d}/4 \cdot R_{\widehat{\ZZ/N}}^{(-1)^d}.
\]
The source $\sS (\CC P^{d+1})$ is calculated in (3-8) of~\cite{Macko-Wegner(2009)} to be a direct sum of severeal copies of $\ZZ$ and several copies of $\ZZ/2$ and the middle term is calculated in the general case in Theorem 3.2 of~\cite{Macko-Wegner(2011)} to be a finite group where we completely understand the $2$-primary torsion and we know the order of the odd-primary torsion.

\begin{lem} \label{lem:rho-factors-through-finite-2-K}
	Let $N = 2^K$. If $k=2l$ the map $[\wrho_{2^K} (d,k)]$ factors as
	\[
		[\wrho_{2^K} (d,k)] \co \ZZ \oplus \sT_{2^K} (d,k) \xra{\proj} \ZZ/{2^K} \oplus \sT_{2^K} (d,k) \xra{[\bar \rho_N (d,k)]} \QQ R_{\widehat{\ZZ/2^K}}^{\pm}/4 \cdot R_{\widehat{\ZZ/2^K}}^{\pm}.		
	\]
\end{lem}

\begin{proof}
	Consider the equation from Proposition~\ref{prop:rho-N-d-k}. 
	The right hand part tells us that if we precompose the map $[\wrho_{2^K}]$ with a projection from $\ZZ (d,k)$ we obtain a map given by the  formula~\eqref{eqn:formula-for-rho-Z-d-k-even} or~\eqref{eqn:formula-for-rho-Z-d-k-odd}. 
	
	Next we notice that upon suitable identification of a subgroup of $\sS (\CC P^{d+1})$ with a direct sum of several copies of $\ZZ$ this formula is identical with the formula for the composition
	\begin{equation} \label{eqn:rho-map-k-zero}
	\sS (\CC P^{d+1}) \xra{\textup{proj} \circ (p_{G}^{S^{1}})^{!} \circ \eta} \widetilde{\sN} (L^{2d+3}) \xra{[\wrho_{2^K} (d+2,0)]} \QQ R_{\widehat{\ZZ/2^K}}^{(-1)^d}/4 \cdot R_{\widehat{\ZZ/2^K}}^{(-1)^d}
	\end{equation}
	from Theorem 4.12 in~\cite{Macko-Wegner(2011)} (see also Theorem 14C.4 in~\cite{Wall(1999)}).

	To see this note that the indexing set for~\eqref{eqn:rho-map-k-zero} is $I_{4}^{S}(d+2,0) = \{ 1, \ldots, e \}$. The appropriate bijection $I_{4}^{S} (d+2,0) \ra I_{4}^{S} (d,k)$ is given by $i \mapsto i+(l-1)$.	
	
	However, in the composition~\eqref{eqn:rho-map-k-zero} every element in the middle term has order which divides $2^K$. Hence the same is true for every element in the image of the composition. But then it also means that every element in the image of $[\wrho_{2^K} (d,k)]$ is of such an order and consequently we have a factorization as claimed.
\end{proof}

\begin{lem} \label{lem:rho-factors-through-finite-M}
	Let $N=M$ odd. If $k=2l$ the map $[\wrho_M (d,k)]$ factors as
	\[
	[\wrho_M (d,k)] \co \ZZ \oplus \sT_{M} (d,k) \xra{\proj} \sT_{M} (d+2,0) \xra{[\bar \rho_M (d,k)]} \QQ R_{\widehat{\ZZ/M}}^{\pm}/4 \cdot R_{\widehat{\ZZ/M}}^{\pm},		
	\]
	where $\sT_M (d+2,0)$ is isomorphic to the image of $[\wrho_{M} (d,k)]$ which is a finite group of order $M^{c+1}$ with $c = \lfloor (d-1)/2 \rfloor$.
\end{lem}

\begin{proof}
	We would like to use the same logic as in the proof of the previous lemma. However, the information that we have about the normal invariants in this case is weaker, so we have to modify the arguments somewhat. 
	
	We still have Proposition~\ref{prop:rho-N-d-k} and the formulas~\eqref{eqn:formula-for-rho-Z-d-k-even} and~\eqref{eqn:formula-for-rho-Z-d-k-odd},
	which upon suitable identification of $\sS (\CC P^{d+1})$ with a direct sum of several copies of $\ZZ$ are identical with the formula for the composition
	\begin{equation} \label{eqn:rho-map-k-zero-M}
	\sS (\CC P^{d+1}) \xra{\textup{proj} \circ (p_{G}^{S^{1}})^{!} \circ \eta} \widetilde{\sN} (L^{2d+3}) \xra{[\wrho_{M} (d+2,0)]} \QQ R_{\widehat{\ZZ/M}}^{(-1)^d}/4 \cdot R_{\widehat{\ZZ/M}}^{(-1)^d}
	\end{equation}
	from Theorem 4.12 in~\cite{Macko-Wegner(2011)} via the same bijection of the indexing sets.
	
	What is different is that in this case we only know that the middle group in the composition~\eqref{eqn:rho-map-k-zero} has cardinality $M^{c+1}$. On the other hand we also know that the map $[\wrho_{M} (d+2,0)]$ is injective by~\cite[14E]{Wall(1999)} (see also Theorem 5.2 in~\cite{Macko-Wegner(2011)}). Since the first map is surjective, the image of the composition has the same cardinality. Now, our desired map $[\wrho_{M} (d,k)]$, as any map, factors through its image. Because of the identification via Proposition~\ref{prop:rho-N-d-k} this image is the same as the image of~\eqref{eqn:rho-map-k-zero-M}. This proves the claim.
\end{proof}

\begin{lem} \label{lem:rho-factors-through-finite-N}
	Let $N = 2^K \cdot M$ with $M$ odd. If $k=2l$ the map $[\wrho_N (d,k)]$ factors as
	\[
	\ZZ \oplus \sT_{N} (d,k) \xra{\proj} \ZZ/{2^K} \oplus \sT_{2^K} (d,k) \oplus \sT_{M} (d+2,0) \xra{[\bar \rho_N (d,k)]} \QQ R_{\widehat{\ZZ/N}}^{\pm}/4 \cdot R_{\widehat{\ZZ/N}}^{\pm},		
	\]
	where $\sT_M (d+2,0)$ is a finite group of order $M^{c+1}$ with $c = \lfloor (d-1)/2 \rfloor$.
\end{lem}

\begin{proof}
	Again, we would like to use the same logic as in the proof of the previous two lemmas, but clearly we have to modify the ideas about the factorization, since two different arguments were used.
	
	Even in this case we still have Proposition~\ref{prop:rho-N-d-k} and the formulas~\eqref{eqn:formula-for-rho-Z-d-k-even} and~\eqref{eqn:formula-for-rho-Z-d-k-odd},
	which upon suitable identification of a subgroup of $\sS (\CC P^{d+1})$ with a direct sum of several copies of $\ZZ$ are identical with the formula for the composition
	\begin{equation} \label{eqn:rho-map-k-zero-N}
	\sS (\CC P^{d+1}) \xra{\textup{proj} \circ (p_{G}^{S^{1}})^{!} \circ \eta} \widetilde{\sN} (L^{2d+3}) \xra{[\wrho_{N} (d+2,0)]} \QQ R_{\widehat{\ZZ/N}}^{(-1)^d}/4 \cdot R_{\widehat{\ZZ/N}}^{(-1)^d}
	\end{equation}
	from Theorem 4.12 in~\cite{Macko-Wegner(2011)} via the same bijection of the indexing sets.
	
	In the present case notice that the middle term is a direct sum of a $2$-primary torsion and a group whose cardinality is $M^{c+1}$. Therefore the image will also be a direct sum of a $2$-primary torsion group and an odd order torsion group. Moreover, due to naturality and the fact that $[\wrho_{M} (d+2,0)]$ was injective in the above lemma, our map is also injective on the odd part and so the cardinality of the odd part summand of the image is $M^{c+1}$. On the $2$-primary part we know that the order of every element divides $2^K$. This proves the claim.
\end{proof}
 
\begin{notation} \label{not:K-N}
	Denote
	\begin{align*}
	K_N := \ker [\wrho_N (d,k)] \quad \textup{and} \quad \bar K_N := \ker [\bar \rho_N (d,k)] \quad & \textup{if} \; k=2l, \\
	K_N := \ker [\wrho_N (d,k)] =: \bar K_N \quad & \textup{if} \; k=2l+1. 
	\end{align*}
\end{notation}


\begin{lem} \label{lem:ses-kernels_N}
	If $k=2l$ we have a short exact sequence 
	\[
	0 \ra N \cdot \ZZ \ra K_N \ra \bar K_N \ra 0.
	\]
\end{lem}

\begin{proof}
	This follows from the commutative ladder
	\[
	\xymatrix{
		0 \ar[r] & N \cdot \ZZ \ar[r] \ar[d] & \ZZ \oplus \sT_{N} (d,k) \ar[r] \ar[d]^{[\wrho_N (d,k)]} & \ZZ/{2^K} \oplus \sT_{2^K} (d,k) \oplus \sT_{M} (d+2,0) \ar[r] \ar[d]^{[\bar{\rho}_N (d,k)]} & 0 \\	
		0 \ar[r] & 0 \ar[r] & \QQ R_{\widehat \ZZ_N}^{\pm}/4 \cdot R_{\widehat \ZZ_N}^{\pm} \ar[r] & \QQ R_{\widehat \ZZ_N}^{\pm}/4 \cdot R_{\widehat \ZZ_N}^{\pm} \ar[r] & 0 
			}
	\]
	with exact rows.
\end{proof}

In order to proceed we need to reformulate known results for $N = 2^{K}$ and $N = M$ odd in terms $[\bar \rho_{2^K} (d,k)]$ and $[\bar \rho_{M} (d,k)]$. 

\begin{prop} \label{prop:K-M}
	We have
	\[
		\bar K_{M} = 0
	\]
	and as a consequence
	\begin{align*}
	K_{M} = \ker [\wrho_{M} (d,k)] = M \cdot \ZZ & \quad \textup{when} \; k=2l, \\
	K_{M} = 0 & \quad \textup{when} \; k=2l+1.
	\end{align*}
\end{prop}

\begin{proof}
	The structure sets $\sS^{s} (L^{2d-1}_{M} \times D^{m})$ were studied in~\cite[\S 3]{Madsen-Rothenberg(1989)}, but their calculation is not suitable for us, so we use a different argument. Note that in the case $k=0$ we have that the map $[\wrho_M (d+2,0)]$ is injective  by~\cite[14E]{Wall(1999)}. When $k=2l$, then arguing as in the proof of Lemma~\ref{lem:rho-factors-through-finite-M} we observe that upon suitable identification of a summand of $\sS (\CC P^{d+1})$ with a direct sum of several copies of $\ZZ$ the map $\rho_M (d,k)$ has the same formula as the composition of $[\wrho_M (d+2,0)]$ with $\textup{proj} \circ (p_{G}^{S^{1}})^{!} \circ \eta$ for $k=0$ and so the map $[\bar \rho_M (d,k)]$ is identified with this map. When $k=2l+1$ then upon a suitable re-indexing the map $[\wrho_{M} (d,k)]$ has the same formula as $[\wrho_M (d+2,0)]$.
\end{proof} 

\begin{prop} \label{prop:K-2-K} For any $k =2l \geq 0$ we have
	\[
	\bar K_{2^K} = \bigoplus_{i=1}^{c_N(d,k)+1} \ZZ/2^{\textup{min} \{ 2i , K \}} \oplus \bigoplus_{i=1}^{c_2 (d,k)} \ZZ/2.
	\]
\end{prop}

\begin{proof}
	This follows from the proofs of Propositions 6.1 and 6.2 in \cite{BMNR(2018)}. The second summand comes from the fact that the formula for $rho$-invariant does not depend on the invariants $\bt_{4i-2}$. 
\end{proof}

Next we would like to combine the two results for $N = 2^K$ and $N = M$ odd. 

\begin{thm} \label{thm:K-bar-N} For any $k \geq 0$ we have
	\[
	(p_{2^K}^{N})^{!} \co \bar K_{N} \xra{\cong} \bar K_{2^K}
	\]
\end{thm}

\begin{proof}
	We use the same logic as the proof of Proposition 6.1 in~\cite{Macko-Wegner(2011)}. When $k=2l$ the situation can be summarized in the diagram
	\[
	\xymatrix{
		\ZZ/2^{K} \oplus \sT_{2^K} (d,k) \ar[d]^{[\bar \rho_{2^K} (d,k)]} & \ZZ/{2^K} \oplus \sT_{2^K} (d,k) \oplus \sT_{M} (d+2,0) \ar[l]_-{(p_{2^K}^{N})^{!}} \ar[r]^-{(p_{M}^{N})^{!}} \ar[d]^{[\bar \rho_N (d,k)]} &  \sT_{M} (d+2,0) \ar[d]^{[\bar \rho_M (d,k)]} \\ 
		\QQ R_{\widehat{\ZZ/{2^K}}}^{\pm}/4 \cdot R_{\widehat{\ZZ/{2^K}}}^{\pm} & \QQ R_{\widehat{\ZZ/N}}^{\pm}/4 \cdot R_{\widehat{\ZZ/N}}^{\pm} \ar[l] \ar[r] & \QQ R_{\widehat{\ZZ/M}}^{\pm}/4 \cdot R_{\widehat{\ZZ/M}}^{\pm}.
	}
	\]
	Just as in that proof the splitting of the middle term into the $2$-primary part and the odd part is used to show 
	\[
	\bar K_N = \big( \ker [\bar \rho_N (d,k)]|_{\ZZ/2^{K} \oplus \sT_{2^K} (d,k)} \big) \oplus \big( \ker [\bar \rho_N (d,k)]|_{\sT_{M} (d+2,0)} \big)
	\] 
	in exactly the same way.
	
	Next it needs to be shown that 
	\[
		\ker [\bar \rho_N (d,k)]|_{\ZZ/2^{K} \oplus \sT_{2^K} (d,k)} \cong  \big( (p_{2^K}^{N})^{!} \big)^{-1} \bar K_{2^K}
	\]
	and
	\[
		\ker [\bar \rho_N (d,k)]|_{\sT_{M} (d+2,0)} \cong  \big( (p_{M}^{N})^{!} \big)^{-1} \bar K_{M} = 0.
	\]
	The second equation follows from the commutativity of the diagram and from Proposition~\ref{prop:K-M}. The first equation is proved in exactly the same way as in~\cite{Macko-Wegner(2011)}, it is basically a purely algebraic statement whose proof involves purely algebraic Lemma 6.2 and 6.3 of that paper. 
	
	When $k=2l+1$ then at the beginning we already have 
	\[
	\sT_N (d,k) \cong \sT_{2^K} (d,k) \oplus \sT_M (d,k)
	\]
	and then we use the same logic.
\end{proof}

\begin{cor} \label{cor:T-N}
When $k=2l$ we have
\[
K_N \cong \ZZ \oplus \bigoplus_{i =1}^{c_{N} (d,k)} \ZZ/{2^{\min\{K,2i\}}} \oplus \bigoplus_{i =1}^{c_2 (d,k)} \ZZ/2.
\]
When $k=2l+1$ we have
\[
K_N \cong \bigoplus_{i =1}^{c_{N} (d,k)} \ZZ/{2^{\min\{K,2i\}}} \oplus \bigoplus_{i =1}^{c_{2} (d,k) + 1} \ZZ/2.
\]
\end{cor}

\begin{proof}
	The case $k=2l+1$ follows from $K_N = \bar K_N$. The numbers $c_N (d,k)$ and $c_{2} (d,k)+1$ are cardinalities of the indexing sets from~\eqref{eqn:torsion-of-ni-general}. The case $k=2l$ goes as follows. By Lemma~\ref{lem:ses-kernels_N} the group $K_N$ is an extension of $\ZZ$ and $\bar K_N$, which is determined by Theorem~\ref{thm:K-bar-N}. Now inspecting the calculations in~\cite{BMNR(2018)} shows that these are essentially done by studying the kernel of $\rho_{2^K} (d,k)$ and then passing to the quotient. There the extension is as we claim. Inspecting the commutative ladder in the proof of Lemma~\ref{lem:ses-kernels_N} proves the general case. 
\end{proof}

\begin{proof}[Proof of Theorem~\ref{thm:main-thm}]
	The proof is the same as the proof of Theorem~1.1 in the paper~\cite{BMNR(2018)} except now we have $L$-groups calculated in Theorem~\ref{L(G)} and the group $\bar T (d,k)$ is denoted $K_N$ in this paper and it is calculated in Corollary~\ref{cor:T-N}. In the case $k=2l+1$ the $\ZZ/2$-summand from Corollary~\ref{cor:T-N} corresponding to $\bt_{4l+2}$ is separated in the statement of Theorem~\ref{thm:main-thm} as being detected by $\bbr_0$.
\end{proof}

\begin{proof}[Proof of Corollary~\ref{cor:higehr-str-sets-L-times-S}]
	The proof is the same as the proof of Corollary 1.2 in the paper~\cite{BMNR(2018)} except the normal invariants $\sN (L_N^{2d-1})$ now contain a $KO$-theory summand, see Theorem 3.2 in~\cite{Macko-Wegner(2011)}. The invariant $\br'''$ was denoted $\bt_{(\textup{odd})}$ in~\cite{Macko-Wegner(2011)}.  
\end{proof}


\section{Final Remarks} \label{sec:final-remarks}

One obvious future direction would be to try to obtain a better geometric description of the invariants $\bbr$ from Theorem \ref{thm:main-thm}.
In \cite{Macko-Wegner(2011)} there was offered an inductive obstruction theoretic description of the corresponding invariants when $m=0$. Such a description works also in the present case with the proof very similar to the case $m=0$, so we refrain from repeating it here and we refer the reader to \cite{Macko-Wegner(2011)}. Of course, even better would be a non-inductive description, but this remains open even in the case $m=0$.

Another improvement that one might seek is a deeper understanding of the $KO$-part of the normal invariants. This involves different techniques, so we postpone it to further projects.

\small
\bibliography{lens-spaces}
\bibliographystyle{alpha}

\end{document}